\documentclass[11pt]{amsart}

\usepackage{
    amsmath,
    amsfonts,
    amssymb,
    amsthm,
    amscd,
    comment,
    enumitem,
    etoolbox,
    gensymb,    
    mathtools,
    mathdots,
    stmaryrd,
    booktabs,
}
\usepackage[usenames,dvipsnames]{xcolor}
\usepackage[all]{xy}

\usepackage[T1]{fontenc}
\usepackage{bbm}                     
\usepackage[colorlinks=true, linkcolor=blue, citecolor=blue, urlcolor=blue, breaklinks=true]{hyperref}

\DeclareFontFamily{OT1}{pzc}{}
\DeclareFontShape{OT1}{pzc}{m}{it}{<-> s * [1.10] pzcmi7t}{}
\DeclareMathAlphabet{\mathpzc}{OT1}{pzc}{m}{it}

\usepackage[capitalize]{cleveref}   

\crefname{defin}{Definition}{Definitions}
\crefname{eg}{Example}{Examples}
\crefname{egs}{Examples}{Examples}
\crefname{convention}{Convention}{Convention}
\crefname{lem}{Lemma}{Lemmas}
\crefname{prop}{Proposition}{Propositions}
\crefname{theo}{Theorem}{Theorems}
\crefname{rem}{Remark}{Remarks}
\crefname{equation}{}{}
\crefname{enumi}{}{}

\newcommand\la{\lambda}
\newcommand\La{\Lambda}

\newcommand\C{\mathbb{C}}

\newcommand\N{\mathbb{N}}

\newcommand\Q{\mathbb{Q}}

\newcommand\hash{\#}

\def\a{\alpha}

\DeclareMathOperator{\Hom}{Hom}

\DeclareMathOperator{\Res}{Res}

\newtheorem{theo}{Theorem}[section]

\newtheorem{cor}[theo]{Corollary}

\theoremstyle{definition}

\numberwithin{equation}{section}
\allowdisplaybreaks

\setenumerate[1]{label=(\alph*)}          

\newtoggle{comments}
\newtoggle{details}
\newtoggle{detailsnote}

\iftoggle{comments}{%
    \usepackage[notref,notcite]{showkeys}   
    \newcommand{\acomments}[1]{
        \ \\
        {\color{red}
            \textbf{AS:} #1
        }
        \ \\
    }
    \newcommand{\pcomments}[1]{
        \ \\
        {\color{blue}
            \textbf{PM:} #1
        }
        \ \\
    }
}{%
    \newcommand{\acomments}[1]{\ignorespaces}
    \newcommand{\pcomments}[1]{\ignorespaces}
}

\iftoggle{details}{%
    \newcommand{\details}[1]{
        \ \\
        {\color{OliveGreen}
            \textbf{Details:} #1
        }
        \\
    }
}{%
    \newcommand{\details}[1]{\ignorespaces}
}

\begin{document}

\title{On a Schur-positive function}

\author{Peter J. McNamara}
\address[]{
    School of Mathematics and Statistics \\
    University of Melbourne \\
    Parkville, VIC, 3010, Australia
}
\urladdr{\href{http://petermc.net/maths}{petermc.net/maths}, \textrm{\textit{ORCiD}:} \href{https://orcid.org/0000-0001-6111-1511}{orcid.org/0000-0001-6111-1511}}
\email{maths@petermc.net}

\begin{abstract}
   We prove Schur-positivity for a family of symmetric functions.
\end{abstract}



\ifboolexpr{togl{comments} or togl{details}}{%
  {\color{magenta}DETAILS OR COMMENTS ON}
}{%
}

\maketitle
\thispagestyle{empty}


\section{Introduction}

Let $e$ be a positive integer and let $\mu_e$ denote the set of $e$-th roots of 1 (in $\C$). For $n \in \N$, define the symmetric polynomial
\begin{equation*}
    W_n^{(e)}(x_1,x_2,\dotsc,x_r)
    = \frac{1}{e^r}
    \sum_{\zeta_1,\dotsc,\zeta_n\in \mu_e} \left( \sum_{i=1}^r \zeta_i \sqrt[e]{x_i} \right)^{en}.
\end{equation*}

Taking the inverse limit over $r$, these define a symmetric function $W_n^{(e)} \in \Lambda$.  In terms of the monomial symmetric functions, we have the expansion
\begin{equation}\label{wmonomial}
    W_n^{(e)}= \sum_{\la\vdash r} \binom{er}{e\la} m_{\la},
\end{equation}
where $\binom{en}{e\la}$ is a multinomial coefficient and $e\la$ denotes the partition obtained from $\la$ by multiplying all parts by $e$.

In \cite[Conjecture 10.7]{spinbrauer}, it was conjectured based on some computer computations that $W_n^{(2)}$ is Schur-positive. In this paper we prove

\begin{theo}
The symmetric function $W_n^{(e)}$ is Schur positive.
\end{theo}

Define coefficients $a_\la$ by
\[
W_n^{(e)} = \sum_{\la\vdash n} a_\la s_\la.
\]
After giving a determinantal formula for $a_\la$, we give three proofs that $a_\la>0$. The first uses the Gessel-Viennot combinatorial interpretation of a determinant of binomial coefficients. The second yields a combinatorial interpretation of the coefficients $a_\la$ as the number of standard Young tableaux of a certain shape. The third yields a representation-theoretic interpretation of the $a_\la$ as dimensions of simple modules over Khovanov-Lauda-Rouquier algebras (henceforth known as KLR algebras, and also known in the literature as quiver Hecke algebras), and provides a natural example of a $S_n$-module whose Frobenius characteristic is equal to $W_n^{(e)}$.

\subsection*{Acknowledgements}

We thank R. Muth, A. Savage and L. Speyer for various comments. In particular we are indebted to L. Speyer for his visit to Melbourne, which encouraged the author to compute the value 61, leading to this work.

\section{A determinantal formula}

We use standard notation for symmetric functions. When $\lambda$ is a partition, $\ell$ can be any integer greater than or equal to the length of $\la$.

Define a ring homomorphism $\varphi^{(e)}:\La\to \Q$ by
\[
\varphi^{(e)}(h_n)= \frac{1}{(en)!}
\]

\begin{theo}\label{varphi}
The function $W_n^{(e)}$ has the Schur expansion
\[
W_n^{(e)} =  (en)!\sum_{\la \vdash n} \varphi^{(e)}(s_\la) s_\la.
\]
\end{theo}

\begin{proof}
We prove that if $f\in \Lambda_n$, then
\[
\langle W_n , f \rangle = (en)! \varphi^{(e)} (f).
\]
By linearity it suffices to prove this when $f=h_\la$. It now follows from the fact that the complete symmetric functions are a dual basis to the monomial symmetric functions, together with the monomial decomposition \cref{wmonomial}. Since the Schur functions are an orthonormal basis, the theorem follows.
\end{proof}

\begin{cor}
The coefficient $a_\la$ is given
by
\begin{equation}\label{coeffdet}
a_\la=(en)! \det \left( \frac{1}{(e(\la_i+j-i))!} \right)_{i,j=1}^\ell
\end{equation}
\end{cor}

\begin{proof}
This is immediate from \cref{varphi} and the Jacobi-Trudi formula.
\end{proof}

%

\section{First proof of positivity}

In the determinant \cref{coeffdet}, we
put the $i$-th row over the common denominator $(e(\la_i+\ell-i))!$ and divide the entries in the $j$-th row by $(e(\ell-j))!$. This yields
\[
a_\la = \frac{(en)! \prod_{j=1}^{\ell-1} (ej)! }{ \prod_{i=1}^\ell (e(\la_i+\ell-i))! } \det \left( {e(\la_i+\ell-i)\choose e(\ell-j)} \right)_{i,j=1}^\ell
\]

By \cite[Theorem 1]{gv}, this determinant of binomial coefficients is equal to the number of $\ell$-tuples of non-intersecting lattice paths, starting at the points $(0,-e(\la_i+\ell-i))$ and ending at the points $(e (\ell-j),e(\ell-j))$ with all steps either east or north. In particular this number is positive, giving the first proof of Schur postivity of $W_n^{(e)}$.

\section{Second proof of positivity}

Given a partition $\la$ of $n$, define the partitions $\mu$ and $\nu$ by
\[
\nu=(e\la_1+(\ell-1)(e-1),e\la_2+(\ell-2)(e-1),\cdots, e\la_\ell),\qquad \mu=((\ell-1)(e-1),(\ell-2)(e-1),\cdots,0)
\]
and define the skew shape $\xi(\la)=\nu/\mu$. In this section, we prove
\begin{theo}
The coefficient $a_\la$ is equal to the number of standard Young tableau of shape $\xi(\la)$.
\end{theo}

\begin{proof}
This is an immediate consequence of the fact \cite{aitken} that for any skew shape $\theta=\alpha/\beta$, we have
\begin{equation}\label{sytformula}
\frac{\hash \{SYT(\theta)\}}{|\theta|!} =
 \det\left( \frac{1}{(\alpha_i+j-i-\beta_j)!} \right)_{i,j}.
\end{equation}
\details{There are two ingredients in the proof of this fact. One is that for any skew shape $\theta$, we have
$$\varphi^{(1)}(s_\theta)=\frac{\hash \{SYT(\theta)\}}{|\theta|!}.$$
This fact is standard, one proof without presupposing knowledge about skew Specht modules is to note that if $\theta=\alpha/\beta$, $\hash \{SYT(\theta)\}$ is the number of paths in Youngs lattice from $\beta$ to $\alpha$. Hence it is equal to
\begin{align*}
\dim\Hom (S^\beta, \Res S^\alpha )&= \dim\Hom(S^\beta,\bigoplus_\nu (S^\beta 
\boxtimes S^\nu)^{\oplus c^{\alpha}_{\beta\nu}}) \\
&= \sum_\nu c_{\beta\nu}^\alpha \dim S^\nu \\
&= \sum_\nu c_{\beta\nu}^\alpha \varphi^{(1)}(s_\nu)|\nu|! \\
&= |\theta|!\varphi^{(1)}(s_\theta).
\end{align*}

The other ingredient is the Jacobi-Trudi formula for skew Schur functions,
namely
\[
s_{\a/\beta} = \det (h_{\a_i + j - i - \beta_j})_{i,j}.
\]
From these ingredients, the proof of \cref{sytformula} is immediate.
}\end{proof}

A combinatorial consequence of the results we have proved is

\begin{cor}
We have
\[
\sum_{\la\vdash n} \# {\rm SYT}(\lambda) \times \# {\rm SYT}(\xi(\lambda)) = \frac{(en)!}{(e!)^n}.
\]
\end{cor}

The $e=1$ case of this result has a combinatorial proof via Robinson-Schensted, so we may ask if there is an $e$-Robinson-Schensted correspondence that produces a bijective proof of this identity.

\section{Third proof of positivity\label{sec:klr}}

For reasons of space we punt the definition of KLR algebras to \cite{klr3}. We will use KLR algebras as in that paper, in particular we work over a field of characteristic zero, and all polynomials $Q_{ij}(u,v)$ are $\pm$ a power of $u-v$. We work in type $A_{e-1}^{(1)}$ and let $\delta$ be the minimal imaginary root. Then there exists a representation $L$ of the KLR algebra $R(\delta)$ whose character satisfies
\[
\operatorname{ch}(L)=[0,1,2,\cdots,e-1].
\]
Then \cite[\S 17]{klr3} constructs for each partition $\lambda\vdash n$ a simple module $L(\la)$ for the KLR algebra $R(n\delta)$, so that we have
\[
L^{\circ n} \cong \bigoplus_{\la\vdash n} L(\la) \otimes S^{\la}
\]
as $R(n\delta)-S_n$-bimodules (here $S^\la$ are the Specht modules).

We will prove
\begin{theo}
For any partition $\la$ of $n$, we have $a_\la = \dim L(\la)$.
\end{theo}
 
\begin{proof}
Let $\mathcal{C}_n$ be the full subcategory of finite dimensional $R(n\delta)$-modules with all Jordan-Holder quotients isomorphic to $L(\la)$ for some partition $\la$. The the results of \cite{klr3} imply the existence of an isomorphism,
\[
\bigoplus_{n=0}^{\infty} K_0(\mathcal{C}_n) \cong \La.
\]
under which the class of the simple module $L(\la)$ gets sent to $s_{\la}$.

The product on the left hand side is given by induction, which implies that $[M]\mapsto \frac{\dim M}{(en)!}$ for $M\in \mathcal{C}_n$ is a ring homomorphism.

We need to show that this ring homomorphism agrees with $\varphi^{(e)}$. Since the complete symmetric functions generate $\Lambda$, it follows from the fact that $L(1^n)$ is one-dimensional, which is easy to see as we can explicitly construct it as the module with character $[(0,1,2,\ldots,e-1)^n]$
\end{proof} 
 
As a consequence, we obtain 
 
\begin{cor}
The Frobenius character of $L^{\circ n}$ is equal to $W_n^{(e)}$.
\end{cor}

We remark also that as we now know the dimensions of the modules $L(\la)$, we can identify them with skew-Specht modules from \cite{skewspecht}. In particular there is a skew-Specht module ${\mathbf{S}}^{\zeta(\la)}$ constructed whose head is $L(\la)$ according to \cite[Theorem B]{skewspecht}. In the characteristic zero case we are considering, this surjection from ${\mathbf{S}}^{\zeta(\la)}$ to $L(\la)$ is an isomorphism by comparing their dimensions. (In general skew-Specht modules are not irreducible, even for purely imaginary skew-Specht modules in characteristic zero).


\bibliographystyle{alphaurl}
\bibliography{schurpositive}

\end{document}